\documentclass[11pt]{article}
\usepackage[utf8]{inputenc}
\usepackage{optidef}
\RequirePackage{amsthm,amsmath,amsfonts,amssymb}
\RequirePackage[numbers]{natbib}
\RequirePackage[colorlinks,citecolor=blue,urlcolor=blue]{hyperref}
\RequirePackage{graphicx}
\usepackage[margin=25pt,font=small,labelfont=bf]{caption}
\usepackage{enumerate}

\newtheorem{thm}{Theorem}[section]
\newtheorem{lemma}[thm]{Lemma}
\newtheorem{prop}[thm]{Proposition}
\newtheorem{cor}[thm]{Corollary}

\newtheorem*{thm*}{Theorem}
\newtheorem*{lemma*}{Lemma}
\newtheorem*{prop*}{Proposition}
\newtheorem*{cor*}{Corollary}
\newtheorem*{conj*}{Conjecture}

\theoremstyle{definition}
\newtheorem{defn}[thm]{Definition}
\newtheorem{ques}[thm]{Question}
\newtheorem{ex}[thm]{Example}
\newtheorem{remark}[thm]{Remark}

\DeclareMathOperator{\tr}{Tr}
\DeclareMathOperator{\gcr}{GCR}
\DeclareMathOperator{\mlt}{MLT}
\DeclareMathOperator{\rank}{rank}
\DeclareMathOperator{\gr}{Gr}

\DeclareMathOperator{\orthP}{OP}

\newcommand{\RR}{\mathbb{R}}
\renewcommand{\SS}{\mathcal{S}}

\usepackage{cleveref}

\title{Maximum likelihood thresholds of generic linear concentration models}
\author{Daniel Irving Bernstein, Steven J.~Gortler, and Louis Theran}
\date{\today}

\usepackage[margin=1in]{geometry}
\begin{document}

\maketitle

\begin{abstract}
    The maximum likelihood threshold of a statistical model is the minimum number of datapoints required to fit the model via maximum likelihood estimation.
    In this paper we determine the maximum likelihood thresholds of generic linear concentration models.
    This turns out to be the number that one might expect from a naive dimension count,
    which is nontrivial to prove given that the maximum likelihood threshold is a semi-algebraic concept.
    We also describe geometrically how a linear concentration model can fail to exhibit this generic behavior.
\end{abstract}

\section{Introduction}

Fitting a model with relatively small datasets is a challenging problem that arises in many applications~\cite[Chapter~18]{hastie},
including modeling networks related to gene regulation~\cite{dobra2004sparse,schafer2005empirical,wu2003interactive} and metabolic pathways~\cite{krumsiek2011gaussian}.
In this paper, we are concerned with a special case of the following.
\begin{ques}\label{ques: general Gaussian models}
    Let $\mathcal{M}$ be an $n$-variate Gaussian model, i.e.~a parameterized family of $n$-variate Gaussian density functions.
    What is the minimum $d$ such that given a dataset consisting of $d$ datapoints,\footnote{assumed to be independently and identically distributed according to a probability distribution that is mutually absolutely continuous with respect to Lebesgue measure on $\mathbb{R}^n$} then, almost surely, there exists some $f \in \mathcal{M}$ of maximum likelihood?
\end{ques}

The $d$ in Question~\ref{ques: general Gaussian models} for a particular $\mathcal{M}$ is called the \emph{maximum likelihood threshold} of $\mathcal{M}$.
One should think of the maximum likelihood threshold as the minimum number of datapoints required to fit a model,
a quantity that becomes relevant in the high-dimensional settings, i.e.~when there are few datapoints relative to the number of variables.
In this paper, we consider the particular case of Question~\ref{ques: general Gaussian models} that takes $\mathcal{M}$ to be a \emph{linear concentration model}. Such models consist of all $n$-variate normal distributions whose inverse covariance matrix (often called the \emph{concentration matrix}) lies in a given linear subspace $L$ of $n\times n$ symmetric matrices; we denote the corresponding model by $\mathcal{M}_L$.
When $L$ is defined by setting certain off-diagonal entries to zero, indexed by the non-edges of some graph,
$\mathcal{M}_L$ is called a \emph{Gaussian graphical model}.

Maximum likelihood thresholds of Gaussian graphical models, and related quantities, have been studied extensively in the algebraic statistics literature.
They were originally introduced by Dempster~\cite{dempster1972covariance}.
Buhl~\cite{buhl1993existence} gave the first bounds on the MLT of a Gaussian graphical model in terms of combinatorial properties of its defining graph.
In particular, he showed that the MLT of a Gaussian graphical model is bound between the defining graph's clique number, and one plus its treewidth.
Uhler~\cite{uhler2012geometry} gave a procedure using Gr\"obner bases to compute an upper bound on the MLT of a Gaussian graphical model.
Gross and Sullivant~\cite{grossSullivantMLT} gave a rigidity-theoretic interpretation of Uhler's upper bound and used rigidity theoretic results to bound and exactly compute the MLTs of certain families of graphs.
Blekherman and Sinn~\cite{blekherman2019maximum} showed that Uhler's upper bound is not sharp in general.
In previous work with Dewar, Nixon, and Sitharam, we gave a rigidity theoretic interpretation of the MLT of a Gaussian graphical model,
and used it to compute the maximum likelihood thresholds of several families of graphs~\cite{bernstein2021maximum,bernstein2022computing}.
The maximum likelihood thresholds of \emph{directed} Gaussian graphical models are also well understood~\cite{drton2019maximum}, 
although they are not linear concentration models.

Given the recent success of rigidity theoretic ideas at demystifying maximum likelihood thresholds of 
Gaussian graphical models,
one might wonder if these ideas can be imported to other models.
This paper is the first step along these lines.
The main result of this paper is Theorem~\ref{thm: main} which gives the maximum likelihood threshold 
$\mathcal{M}_L$ where $L$ is a \emph{generic} linear subspace of symmetric matrices.
In short, it says that for a generic $L$ the maximum likelihood threshold
can be computed by a naive dimension count, using no more information about $L$ than its dimension.
This is the nicest possible answer one could 
hope for, but since we are working over the real field, the result does not immediately follow from 
algebro-geometric considerations.
While generic linear concentration models are not used in practice,
understanding their behavior gives us a baseline that we can compare other families of linear 
concentration models against.
We also expect that, by using tangent spaces, ideas developed for understanding maximum likelihood 
thresholds of linear concentration models
can be used to understand models obtained via smooth constraints on the inverse 
covariance matrix.

One appealing feature of our approach is that we are explicit about which linear subspaces fail to attain this naive bound;
these are the subsets of the Grassmannian described in Lemmas~\ref{lem:psd2} and~\ref{lem:hard2}.
Linear concentration models that appear in practice often exhibit behavior that is different from the generic 
one, because they are defined by linear constraints with a special structure.  Notably, while the number of 
edges in a graph $G$ determines the dimension of the associated Gaussian graphical model, the MLT 
depends in a subtle way on the combinatorial structure of $G$ \cite{blekherman2019maximum,bernstein2021maximum}, 
and it can be much lower than naive dimension (or edge) counting would predict.
Lemmas~\ref{lem:psd2} and~\ref{lem:hard2} can then be interpreted as describing geometric conditions on 
$L$ that lead to more complicated behavior of the MLT.

We begin with Theorem~\ref{thm: existence of optimum criteria}, which establishes two conditions that are equivalent to the existence of a maximum likelihood estimate in $\mathcal{M}_L$ given a particular dataset.
We then use Theorem~\ref{thm: existence of optimum criteria} to establish a lower bound for the maximum likelihood threshold of $\mathcal{M}_L$ in Proposition~\ref{prop: hard direction}, and then again to establish that this lower bound is also an upper bound in Proposition~\ref{prop: easy direction}.
Section~\ref{sec:gcr} studies an easily computable upper bound on the maximum likelihood threshold called the 
\emph{generic completion rank}. This generalizes a notion of the same name for Gaussian graphical models.
Using this  language, Theorem~\ref{thm: main} says that the generic completion rank of a generic linear space is equal to its maximum likelihood threshold.

\subsection{Notation}
In what follows, $\mathcal{S}^n$ denotes the vector space of $n\times n$ real symmetric matrices,
$\mathcal{S}^n(d)$ denotes the closed subvariety of matrices of rank at most $d$,
$\mathcal{S}_0^n(d)$ denotes the constructible set of matrices of rank exactly $d$,
$\mathcal{S}^n_+$ denotes the closed subcone of positive semidefinite (PSD) matrices,
$\mathcal{S}^n_{++}$ denotes the open subcone of positive definite matrices, and 
$\mathcal{S}_+^n(d)$ denotes the set of 
PSD matrices of rank at most $d$.
A \emph{linear concentration model} is a set of the form $L \cap \mathcal{S}^n_{++}$
where $L \subseteq \mathcal{S}^n$ is a linear space.
The identity matrix of size $k\times k$ is denoted $I_k$.

\section{Motivation in detail}\label{sec:motivation}
Given datapoints $x_1,\dots,x_d \in \mathbb{R}^n$ with sample covariance matrix\footnote{this is the matrix $\frac{1}{d}\sum_{i=1}^d (x-\mu)(x-\mu)^T$ where $\mu = \frac{1}{d} \sum_{i=1}^d x_i$} $S$
the \emph{maximum likelihood estimate in the linear concentration model for $L$}
is the inverse of the solution, if it exists, 
to the following optimization problem (see e.g.~\cite[p. 632]{hastie})

\begin{mini}|l|
   {K}{\tr(SK) - \log \det K}{}{\label{eq: mle optimization}}
  \addConstraint{K \in \mathcal{S}^n_{++} \ {\rm and} \ K \in L}.
\end{mini}

The \emph{maximum likelihood threshold} of a linear space $L \subseteq \mathcal{S}^n$, denoted $\mlt(L)$, is the minimum $d$
such that \eqref{eq: mle optimization} has a solution for all generic $S \in \mathcal{S}^n_+$ of rank $d$.
Note that if $L$ does not contain a positive definite matrix, then $\mlt(L)$ is infinite
so we restrict our focus to linear spaces that contain a positive definite matrix.
The set $\mathcal{S}^n_{++}\cap L$ is called the \emph{feasible region} its elements are called \emph{feasible points}.
Given an inner product space $V$ and a linear subspace $L \subseteq V$,
$L^\perp$ denotes the orthogonal complement of $L$
and $\pi_L:V \rightarrow L$ denotes orthogonal projection onto $L$ with kernel $L^\perp$.
Here, we are mainly interested in the $\binom{n+1}{2}$-dimensional space 
$\mathcal{S}^n$ with the inner product of $\Omega, S\in \mathcal{S}^n$ given by 
$\tr(\Omega S)$. If $\Omega$ and $S$ are both PSD, then $\tr(\Omega S) = 0$ if and only if $\Omega S = 0$

\begin{thm}\label{thm: existence of optimum criteria}
    Let $S \in \mathcal{S}^n_{+}\setminus\{0\}$ and let $L \subseteq \mathcal{S}^n$ be a linear subspace 
    such that $L \cap \mathcal{S}^n_{++} \neq \emptyset$.
    Then the following are equivalent:
    \begin{enumerate}[(1)]
        \item there is a (unique) optimum solution to the problem \eqref{eq: mle optimization},
        \item there exists a matrix $K \in \mathcal{S}^n_{++}$ such that $S - K^{-1} \in L^\perp$, and
        \item there does not exist a non-zero matrix $\Omega \in \mathcal{S}^n_{+}\cap L$ such that 
        $\tr(\Omega S) = 0$ (equivalently, that $\Omega S = 0$).
    \end{enumerate}
    In this case, the optimum is the unique matrix $K$ satisfying $K \in L \cap \mathcal{S}^n_{++}$ and $S - K^{-1} \in L^\perp$.
\end{thm}

\begin{proof}
    $(1)\implies (2)$: 
    Let $K$ be the assumed optimal point.
    Since the objective function is strictly convex, $K$ is the unique optimizer.
    Since $K$ is a feasible point, $K \in L \cap \mathcal{S}^n_{++}$.
    The gradient of the objective function at $K$ is $S-K^{-1}$.  As $K$ is in the relative
    interior of the feasible region, the principle of Lagrange multipliers 
    says that the gradient must be orthogonal to the constraints at an optimal 
    value.  Hence, $S - K^{-1}\in L^\perp$, as desired.
    
    $(2) \implies (3)$: Let $K$ satisfy $(2)$ and let $\Omega  \in \mathcal{S}^n_+ \cap L$ such 
    that  $\tr(\Omega S) = 0$ be given.
    From $(2)$, $\Omega $ and the gradient $S - K^{-1}$ are orthogonal.  We then compute 
    \[
        0 = \tr(\Omega (S - K^{-1})) = \tr(\Omega S) - \tr(\Omega K^{-1}) = \tr(\Omega K^{-1})
    \]
    using bilinearity of the inner product.
    Since $\Omega$ and $K^{-1}$ are PSD, the above implies that the column span 
    of $K^{-1}$ lies in the kernel of $\Omega$.  Since $K^{-1}$ is full rank, it follows
    that $\Omega$ has an $n$-dimensional kernel, so $\Omega=0$.
        
    $(3) \implies (1)$: 
    Suppose (3) holds and let $K \in L \cap \mathcal{S}^n_{++}$ (recall that we assume that such a 
    $K$ exists). Define $L_0$ to be the intersection of $L$ and the unit sphere in $\mathcal{S}^n$, i.e.
    \[
        L_0 := \{\Omega \in L : \tr(\Omega^2) = 1\}.
    \]
    As $K$ is  positive definite, for each 
    $\Omega \in L_0$ there exists $\lambda > 0$ such that $K + \lambda \Omega \in L \cap \mathcal{S}^n_{++}$;
    define $\lambda_\Omega$ to be the supremum over all such $\lambda$.
    Define $f(K) := \tr(SK) - \log\det K$.
    We will show that for any $\Omega \in L_0$,
    \begin{equation}\label{eq: diverges to positive infinity}
        f(K + \lambda \Omega) \rightarrow \infty \quad {\rm as} \quad \lambda \rightarrow \lambda_\Omega.
    \end{equation}
    It will then follow that $f(K + \lambda \Omega) > f(K)$ outside of a compact neighborhood of $K$.
    The restriction of $f$ to this neighborhood has a global minimum (as a continuous function on a 
    compact set), which, by construction, is a global minimum of $f$.  In other words, the problem
    \eqref{eq: mle optimization} has an optimal solution.  The rest of the proof establishes 
    Equation \eqref{eq: diverges to positive infinity} by a case analysis.
    
    If $\Omega \notin \mathcal{S}^n_{+}$ then $\lambda_\Omega$ is finite, since 
    for large enough $t$, the signature of $K + t\Omega$ will agree with that of $\Omega$.
    Hence $K + t_\lambda \Omega$ lies on the boundary of the PSD cone $\mathcal{S}^n_+$, 
    making it singular.
    Hence, for $0 < \lambda < \lambda_\Omega$, the term $\tr(S(K + \lambda \Omega))$ 
    is bounded. Then~\eqref{eq: diverges to positive infinity} follows since
    \[
        \log \det(K + \lambda \Omega)\to -\infty\qquad \text{as $\lambda\to \lambda_\Omega$}.
    \]    
    If $\Omega \in \mathcal{S}^n_{+}$ then $\lambda_\Omega = \infty$, since for any $\lambda > 0$, 
    $K + \lambda \Omega$ is the sum of a PD and a PSD matrix, which is again PD.
    Since $\det(K + \lambda \Omega)$ is a polynomial in $\lambda$ of degree $n$, it follows that
    \[
        f(K + \lambda \Omega) = \tr(SK) + \lambda \overbrace{\tr(S\Omega)}^{> 0} - O(\log(\lambda)).
    \]
    This diverges to $\infty$ as $\lambda \rightarrow \infty$.
\end{proof}

\section{Generic linear subspaces}\label{sec:generic}
In this section we derive the maximum likelihood threshold of a generic linear subspace $L \subseteq \mathcal{S}^n$.
A note on our use of geometric and topological language is in order.
Let $V$ be a finite-dimensional vector space.
Throughout this paper, $V$ will either be $\mathbb{R}^n$, equipped with the standard inner product,
or $\mathcal{S}^n$, equipped with the trace-norm inner product, i.e.~where the inner product of matrices $A$ and $B$ is defined to be $\tr(AB)$.

The \emph{Zariski topology} on $V$ is the topology whose closed sets are varieties.
We use the convention that varieties are the vanishing loci of systems of polynomials.
In particular, they need not be irreducible.
The \emph{Euclidean topology} on $V$ is the more familiar topology derived from the metric associated to the inner product on $V$.
The topology induced on the projective space $\mathbb{P}(V)$ by the Zariski/Euclidean and the usual quotient map will be referred to by the same name.

The Zariski/Euclidean topology on any subset $X$ of $V$ or $\mathbb{P}(V)$ is the subspace topology on $X$ induced by the Zariski/Euclidean topology on $V$ or $\mathbb{P}(V)$.
Topological statements that do not explicitly specify a topology should be interpreted in the Euclidean topology.
We will always explicitly say so when we are using the Zariski topology.
For any subset $X$ of $V$ or $\mathbb{P}(V)$, the Zariski topology is coarser than the Euclidean topology.

Recall that the \emph{interior} of a subset $X$ of any topological space is the union of all open sets contained in $X$.
Now let $X \subseteq V$.
The \emph{affine hull} of $X$ is the intersection of all affine subspaces (i.e.~translations of linear subspaces) containing $X$.
The affine hull of $X$ comes equipped with the subspace topology induced by the Euclidean topology on $V$.
The \emph{relative interior} of $X$ is the interior of $X$ in this subspace topology.
Every convex set has non-empty relative interior.

\begin{defn}\label{defn: genericity}
    Let $X$ be an irreducible semi-algebraic set. We say that a statement is \emph{generically} true on $X$,
    or holds at \emph{generic points} of $X$ if it is true on a nonempty Zariski-open subset of $X$.
\end{defn}

Definition~\ref{defn: genericity} requires $X$ to be irreducible so that
the set of points where a generically true statement fails is nowhere-dense in $X$.
The following lemma gives a sufficient condition for a semi-algebraic subset of a semi-algebraic
set to be contained in a Zariski-closed subset.
\begin{lemma}\label{lemma: generic means you can wiggle stuff}
    Let $X$ be a semi-algebraic subset of some real vector space $V$ and let $E \subseteq X$ be a semi-algebraic subset.
    If for all $y \in E$ and for every Euclidean neighborhood $U$ of $y$ in $X$,
    there exists some $x \in X \setminus E$, then $E$ is contained in a Zariski-closed proper 
    subset of $X$.
\end{lemma}
\begin{proof}
First assume $X$ is smooth.
As $E$ is semialgebraic, $E$ can be stratified into finitely many smooth semialgebraic sets $E_1,\dots,E_k$~\cite{basuAlgorithmsInRealAlgebraicGeometry}.
Our assumptions imply that $E$ is nowhere dense in $X$;
the same therefore holds for each $E_i$.
Since each $E_i$ is a nowhere-dense smooth submanifold of $X$,
each $E_i$ has strictly lower dimension than $X$.
Thus $E$ is of lower dimension
than $X$.
Meanwhile, by the results in \cite[Section 2.8]{boch}, 
the real Zariski closure $Z$ of $E$ in $V$ has the same dimension as $E$. Hence 
$X \nsubseteq Z$, and so $Z \cap X$ is the desired subset.

If $X$ is not smooth, then the non-smooth locus of $X$ is contained in a semi-algebraic subset $W'$ of $X$ of lower dimension
(this follows from Proposition~5.53 in~\cite{basuAlgorithmsInRealAlgebraicGeometry}).
We will denote the Zariski closure of $W'$ in $X$ by $W$, which has the same dimension as $W'$.
Then applying the previous argument to intersection of the smooth locus of $X$ with $E$ gives us a Zarkiski-closed subset $Z$ of $X$. Then $Z \cup W$ contains $E$ and is Zariski-closed in $X$.
\end{proof}

In order to talk about generic linear spaces (in the sense of Definition~\ref{defn: genericity}),
we need a way to view each linear space as an element of some semi-algebraic set.
A \emph{Grassmannian}, defined and explained below, is that semi-algebraic set.
A \emph{generic linear space} is a generic point of some Grassmannian.
\begin{defn}
    Let $V$ be a vector space and let $d \le \dim(V)$ be an integer.
    The \emph{Grassmannian of $d$-planes in $V$}, denoted $\gr(d,V)$ is the 
    set of all $d$-dimensional linear subspaces of $V$.
\end{defn}
Given a field $\mathbb{F}$ and an $\mathbb{F}$-vector space $V$,
we briefly explain one way to view the Grassmannian $\gr(d,V)$ as a projective variety.
For more details, see, e.g., \cite[Lecture~6]{harris2013algebraic}.  Let $L$ 
be a $d$-dimensional linear subspace of an $n$-dimensional vector space $V$
over a field $\mathbb{F}$ with a distinguished basis $\{e_1, \ldots, e_n\}$.  
If $v_1, \ldots, v_d$ is a basis of $L$, the vector of the $\binom{n}{d}$
minors of the $d\times n$ matrix $M$ that has as its rows the coordinates of 
the $v_i$ in the basis $\{e_j\}$ are called the \emph{Plücker coordinates} of $L$.
Changing basis of $L$ amounts to replacing $M$ by $BM$ for some invertible $d\times d$
matrix $B$, so changing basis of $L$ scales the Plücker coordinates.  Conversely, 
a basis of $L$ can be reconstructed from its Plücker coordinates, and the map that 
sends $L$ to its Plücker coordinates is a closed embedding $\gr(d,V)\to \mathbb{P}\left(\mathbb{F}^{\binom{n}{d}}\right)$.
In what follows, we are interested in the case where $\mathbb{F}=\mathbb{R}$, which makes $\gr(d,V)$, 
additionally, a semi-algebraic set.

In what follows, we will use the idea of Lemma \ref{lemma: generic means you can wiggle stuff} 
on subsets of a Grassmannian.  Since it is easier to describe perturbations of
bases instead of points on the Grassmannian, we provide a lemma that lets us 
switch back and forth.

\begin{lemma}\label{lem: can wiggle basis}
Let $E\subseteq \gr(d,V)$ be a semi-algebraic subset of the Grassmannian of $d$-dimensional 
subspaces of an $n$-dimensional real inner-product space $V$.
Suppose for each $L\in E$, there exists a basis $v_1, \ldots, v_d$ of $L$,
such that every neighborhood of $(v_1, \ldots, v_d)\in V^d$ contains a point $(w_1, \ldots, w_d)$ so that 
the span of $\{w_1,\dots,w_d\}$ is not in $E$.
Then $E$ is contained is a Zariski-closed proper subset of $\gr(d,V)$.
\end{lemma}
\begin{proof}
    Let $L \in E$ and let $B$ be a Euclidean neighborhood of $L$ in $\gr(d,V)$. 
    Let $p: V^d \rightarrow \gr(d,V)$ be the map that sends a point $(v_1, \ldots, v_d)\in V^d$ to the Plücker coordinates of its span.
    Because $p$ is continuous, the set $U := p^{-1}(B)$ is an open subset of $V^d$ that contains every basis of $L$.  By 
    hypothesis there is some basis $(v_1, \ldots, v_d)$ of $L$ so that any neighborhood 
    $U'$ of it contains a point $(w_1, \ldots, w_d)\in V^d$ with the property that 
    $p(w_1, \ldots, w_d)$ is not in $E$.  In particular, $B = p(U) \nsubseteq E$.

    As $L\in E$ and $B$ were arbitrary, $E$ has empty Euclidean interior in $\gr(d,V)$ and is therefore of lower dimension.
    The result then follows as in the proof of Lemma~\ref{lemma: generic means you can wiggle stuff}.
\end{proof}

\subsection{Lower bound on the MLT}
The main result  of this subsection is Proposition~\ref{prop: hard direction} which lower-bounds
the maximum likelihood threshold of $\mathcal{M}_L$ for generic $L$.

To begin, we need the description of the faces of the PSD cone.
\begin{lemma}\label{lemma: faces of PSD cone}
    Given a $d$-dimensional linear subspace $V$ of $\mathbb{R}^n$, define
    \[
        F_V := \{A \in \mathcal{S}^n_+: Ax = 0 \ {\rm for \ all} \ x \in V\}.
    \]
    Then $F_V$ is a face of $\mathcal{S}^n_+$ and is linearly isomorphic to $\mathcal{S}^{n-d}_+$ (so in particular of dimension $\binom{n-d+1}{2}$).
    Moreover, every face of $\mathcal{S}^n_+$ is of this form.
\end{lemma}
\begin{proof}
See \cite[Theorem A.2]{pataki2000geometry}.
\end{proof}

The following key lemma is a 
slight variation of results that can be found in~\cite{barvinok1995problems} and~\cite{pataki1996cone}.
\begin{lemma}
\label{lemma:psd1}
    Let $d\le n-1$.
    Let $L \subseteq \mathcal{S}^n$ be a linear space such that 
    $L \cap \mathcal{S}^n_{++} \neq \emptyset$.
    Assume $L$ has dimension $m$ where
\[
    m \ge dn-\binom{d}{2} -n + d + 1.
\]
    Then $L$ contains a nonzero PSD matrix of rank $n-d$ or less.
\end{lemma}
\begin{proof}
By assumption, $L$ contains a non-zero  PSD matrix.
Let $s$ be the smallest, non-zero, rank of a PSD matrix in 
$L$. Suppose $s \ge (n-d+1)$ (which implies $s \ge 2$).
Let $\Omega$ be such a matrix of rank $s$. 
Let $F$ be the unique face of $\mathcal{S}^n_+$ containing $\Omega$ in its relative interior.
Lemma~\ref{lemma: faces of PSD cone} implies that $\dim(F) = \binom{s+1}{2} \ge \binom{n-d+2}{2}$.
Therefore $\dim(F) + \dim(L)$ is at least
\[
    dn - \binom{d}{2}  -n +d  +1 + \binom{n-d+2}{2} = 
    \binom{n+1}{2} + 2.
\]
Since $L$ intersects the relative interior $F$, this implies $\dim(F \cap L) \ge 2$.
Thus $L$ contains an affine line $l$ that does not go through the origin and 
has a one dimensional intersection with $F$. 
Meanwhile, since $F$ is
convex, line-free and closed, the line $l$ must intersect the relative boundary
of $F$ (in at least one of its two directions) at a non-zero
point.
But any such intersection point has rank at most $s-1$, contradicting the minimality of $s$.
\end{proof}

\begin{lemma}
\label{lem:psd2}
Let $n$ be a positive integer, let $d \le n-1$ and let
\[
    m \ge dn-\binom{d}{2} -n + d + 1.
\]
Let $E \subseteq \gr(m,\mathcal{S}^n)$ be the set of linear spaces that contain a positive definite matrix
but do not contain a PSD matrix of rank $n-d$.
Then $E$ lies inside a Zariksi-closed proper subset of $\gr(m,\mathcal{S}^n)$.
\end{lemma}
\begin{proof}
Let $L\in E$.  By Lemma \ref{lemma:psd1}, $L$ contains a non-zero PSD matrix $\Omega_1$ of rank at most $n-d$. 
Suppose that the rank of $\Omega_1$ has than $n-d-k$ for some $k \ge 1$.
Let $X \in \mathbb{R}^{n\times k}$ be a matrix of rank $k$ whose columns all lie in the nullspace of $\Omega_1$.
Then $\Omega_1 + tXX^T$ is PSD of rank $n-d$ for all $t > 0$.
Now extend $\{\Omega_1\}$ to a basis $\{\Omega_1, \ldots, \Omega_m\}$ of $L$ (it is here we use that $\Omega_1$ is 
non-zero).  For almost all $t > 0$, $\Omega_1 + tXX^T$ lies outside the linear span of
$\{\Omega_2, \ldots, \Omega_m\}$.  

Now fix a neighborhood $U \subseteq \gr(m,\mathcal{S}^n)$ of $L$.  For  sufficiently small $t > 0$, the linear span of 
$\{\Omega_1 + tXX^T, \ldots, \Omega_m\}$ has dimension $m$, lies in $U$ and 
contains the non-zero PSD matrix of rank $n - d$, $\Omega_1 + tXX^T$.  Hence, 
$U$ contains a point not in $E$, and the result follows from Lemma~\ref{lem: can wiggle basis}.
\end{proof}

\begin{ex}\label{ex: first bad thing}
We give an example where the $E$ in Lemma~\ref{lem:psd2} is not empty.
Take $d = 1$, let $n$ be even and let $L'$ be a linear subspace of $\mathcal{S}^{n/2}$ of dimension $(n+2)/2$.
Define $L$ to be the following linear subspace of $\mathcal{S}^n$
\[
    L = \left\{
    \begin{pmatrix}
        A & 0 \\
        0 & A
    \end{pmatrix} \vert
    A \in L'
    \right\}.
\]
Then $L$ has dimension $n+2$ which abides by the lower bound in Lemma~\ref{lem:psd2}.
If $L'$ contains a positive definite matrix, then so does $L$.
However, $L$ does not contain \emph{any} matrices of rank $n-d = n-1$ since it only contains matrices of even rank.
\end{ex}

We note that the behavior observed in Example~\ref{ex: first bad thing} cannot happen when $L$ is defined by setting certain off-diagonal entries to zero, i.e.~the linear spaces defining Gaussian graphical models.
This is because such linear spaces include all diagonal matrices.

We now define some specialized notation that we will need throughout the rest of this section.
Let $\Omega_0$ be a matrix in $\mathcal{S}^n_0(n-d)$.
We denote by $Z(\Omega_0)$ the set of all matrices in 
$\mathcal{S}^n$ and with the same column span of
$\Omega_0$.
The Zariski closure of $Z(\Omega_0)$ is a linear space which we denote by $\overline Z(\Omega_0)$.
$\overline Z(\Omega_0)$ is the set of matrices 
in $\mathcal{S}^n$ whose
column span is contained in that of $\Omega_0$.

Let $S_0$ be a matrix in $\mathcal{S}^n_0(d)$.
We denote by $Q(S_0)$ the set of all matrices in 
$\mathcal{S}^n$   whose kernel is the column 
span of $S_0$. These will all have rank $n-d$.
Moreover, the Zariski closure of $Q(S_0)$ is a linear space which we denote by $\overline Q(S_0)$ and 
it consists of the matrices in $\mathcal{S}^n$   whose kernel contains the column 
span of $S_0$.
Note that if $S_0$ has rank $d$, $\Omega_0$
has rank $n-d$ and $\Omega_0 S_0=0$, then $Q(S_0)=Z(\Omega_0)$.

\begin{lemma}\label{lemma: dimension of symmetric matrices with fixed kernel}
    Let $W$ be a linear subspace of $\mathbb{R}^n$ of dimension $k$.
    The set $\mathcal{K}$ of $n\times n$ symmetric matrices whose columns span $W$ has dimension $k^2-\binom{k}{2}$.
\end{lemma}
\begin{proof}
    The dimension of $\mathcal{K}$ is the same as that of $\mathcal{K}\cap \mathcal{S}^n_+$,
    which by Lemma~\ref{lemma: faces of PSD cone} implies the lemma.
\end{proof}

Lemma~\ref{lemma: dimension of symmetric matrices with fixed kernel} implies that the dimension of $Z(\Omega_0)$ is
\[
    (n-d)^2 - \binom{n-d}{2} = \binom{n-d+1}{2}.
\]

We also need the following linear algebra lemma, which we state without proof.

\begin{lemma}\label{lem: transverse linear spaces with intersection}
Let $V$ be an $N$-dimensional vector space, $W\subseteq V$ an $m$ dimensional linear subspace
of $V$, $x\in W$ a non-zero vector.  If $n$ is such that $m + n > N$, then 
there is an $n$-dimensional subsapce $X\subseteq V$ so that $x\in X$ and 
$X$ intersects $W$ transversally.
\end{lemma}

The next lemma says for generic $L$,
we will see a transverse intersection between
$L$ and $Z(\Omega)$ for 
some  PSD $\Omega \in \mathcal{S}^n_0(n-d)\cap L$.
This transversal intersection is ultimately
all we will need to complete the lower
bound of this section.
\begin{lemma}
    \label{lem:hard2}
    Let $d,m,n$ be integers satisfying $1 \le d \le n$ and $m\ge dn-\binom{d}{2} + 1$.
    Let $E \subseteq \gr(m,\mathcal{S}^n)$ consist of all linear spaces $L$ satisfying
    \begin{enumerate}
        \item $L$ contains a PSD matrix of rank $n-d$, and
        \item $Z(\Omega)$ and $L$ intersect non-transversally for all PSD $\Omega \in L$ of rank $n-d$.
    \end{enumerate}
    Then $E$ lies inside a Zariski-closed proper subset of $\gr(m,\mathcal{S}^n)$.
\end{lemma}
\begin{proof}
Let us first fix a PSD $\Omega = \SS^n_0(n-d)$ and define 
\[
    V := \{ L\in  \gr(m,\mathcal{S}^n) : \Omega \in L\}.
\]
\textbf{Claim:} The subset of $V$ that has non-transversal intersection with 
$Z(\Omega)$ is nowhere dense in $V$.
\textbf{Proof:} The set $V$ is an irreducible algebraic subset of $\gr(m,\mathcal{S}^n)$ since it can be parameterized by a Zariski-open subset of $(\mathcal{S}^n)^{m-1}$.
Indeed, each linearly independent $(m-1)$-element subset of $\mathcal{S}^n$ that does not contain $\Omega$ in its linear span determines the Pl\"ucker coordinates of an element of $V$, and every element of $V$ is spanned by $\Omega$ and one such $(m-1)$-element subset of $\mathcal{S}^{n}$.
Since 
\[
    \binom{n-d+1}{2} + m \ge \binom{n-d+1}{2} + dn - \binom{d}{2} + 1 > \binom{n+1}{2}
\]
and the first term is the dimension of $Z(\Omega)$, Lemma \ref{lem: transverse linear spaces with intersection}
implies that there is an $L\in \gr(m,\SS^n)$ so that $\Omega\in L$ and $L$ and $Z(\Omega)$
meet transversally.  Hence, the subset of $V$ that has non-transversal intersection with 
$Z(\Omega)$ is proper.
As non-transversal intersection with $Z(\Omega)$ is an 
algebraic condition on $L$, the subset of $V$ consisting 
of $L$ with non-transversal intersection with $Z(\Omega)$ 
is of lower dimension and, 
since $V$ is parameterized by a smooth variety, the subset
is nowhere dense in $V$, proving the claim.

Returning to the lemma, we suppose that $L\in E$ and fix a neighborhood $U$ of $L$ in 
$\gr(m,\mathcal{S}^n)$.  Fix 
a PSD $\Omega\in \SS_0^n(n-d)\cap L$.  By the above claim,
$V\cap U$ contains an $L'$ that intersects $Z(\Omega)$ transversally.
In particular, $L'\notin E$.  Hence, $E$ is nowhere dense in $\gr(m,\mathcal{S}^n)$.
By the Tarski--Seidenberg principle, $E$ is semi-algebraic, so 
using 
Lemma~\ref{lemma: generic means you can wiggle stuff}  
we are done.
\end{proof}

\begin{remark}
There is also a more algebraic way to 
prove 
Lemma~\ref{lem:hard2}.
We start by using Kleiman's 
Theorem~\cite{Kleiman} to see that a generic $L$ will intersect $\mathcal{S}^n(n-d)$
transversally.  Using transversality, we can show that the 
intersection $Y$ is irreducible.  A dimension count and Lemma \ref{lem:psd2} 
then imply that the PSD matrices 
in $Y$ are a Zariski dense subset.
Then one argues (with a proof similar to that of  Lemma~\ref{lem:hard2}) 
that for generic $L$, there is 
an  $\Omega \in L \cap \mathcal{S}^n(n-d)$ so that
$Z(\Omega)$ intersects $L$ transversally.
The same then holds for a Zariski open subset of $L\cap \mathcal{S}^n(n-d)$
and thus at some PSD $\Omega$.
We do not pursue this to avoid having to prove that $Y$ is irreducible.
\end{remark}

The last ingredient we need to prove the main result of this section is a distance metric on the Grassmannian.
There are multiple ways to do this; we will describe one.
For more about metrics on the Grassmannian, see e.g.~\cite{bendokat2024grassmann,kozlov2000geometry}.

Recall that an \emph{orthogonal projection operator} on a Hilbert space $V$ with inner product $\langle \cdot, \cdot \rangle$ is a linear map $\pi: V \rightarrow V$ such that $\pi^2 = \pi$ and $\langle x, \pi y \rangle = \langle \pi x, y \rangle$ for all $x, y \in V$.
Given a linear subspace $L \subseteq V$, the orthogonal projection operator sending each $x \oplus y \in L \oplus L^\perp = V$
to $x$ is denoted $\pi_L$.  
We denote the set of orthogonal projection operators on $V$ of rank $d$ by $\orthP(d,V)$.
The following lemma collects some well-known elementary facts about orthogonal projection operators.

\begin{lemma}\label{lem: orthogonal projection operators}
    If $V$ is a Hilbert space then the following hold.
    \begin{enumerate}[\rm (a)]
        \item The function $m: \gr(d,V)\times \gr(d,V)\rightarrow \mathbb{R}$ given by
        $m(L,K) := \|\pi_L - \pi_K\|$ is a metric on $\gr(d,V)$, where $\|\cdot\|$ is the operator norm induced by the inner product on $V$.
        \item For each positive integer $d$, the map $L \mapsto \pi_L$ induces a bijection $\gr(d,V) \rightarrow \orthP(d,V)$.
    \end{enumerate}
\end{lemma}
\begin{proof}
    The first statement is immediate.
    We now prove the second.
    Let $\pi: V \rightarrow V$ be an orthogonal projection operator and define $L := \pi(V)$.
    If $x \in \ker(\pi)$, then $x \in L^\perp$ since, for all $\pi(y)\in L$,
    \[
        \langle x, \pi(y) \rangle = \langle \pi(x) , y\rangle = \langle 0, y\rangle = 0.
    \]
    Conversely, if $x \in L^\perp$, then for all $y \in V$
    \[
        0 = \langle x, \pi(y) \rangle = \langle \pi(x), y \rangle
    \]
    which implies $\pi(x) = 0$. Thus $L^\perp = \ker(\pi)$.
    Now let $x \in V$.
    There exist unique $y \in L$ and $z \in \ker(\pi)$ such that $x = y + z$.
    Let $w \in V$ be such that $\pi(w) = y$.
    Thus
    \[
        \pi(x) = \pi(y + z) = \pi(y) = \pi(\pi(w)) = \pi(w) = y = \pi_L(y + z) = \pi_L(x).
    \]
    Therefore every orthogonal projection operator is of the form $\pi_L$ for some linear subspace $L \subseteq V$
    and so the map in question is surjective.
    For injectivity, note that if $L_1,L_2$ are linear subspaces of $V$ such that $\pi_{L_1} = \pi_{L_2}$
    then $L_1 = \pi_{L_1}(V) = \pi_{L_2}(V) = L_2$. Thus the second statement is proven.
\end{proof}

We will call the metric on $\gr(d,V)$ given by Lemma~\ref{lem: orthogonal projection operators} the \emph{projector metric}, 
and use the same notation $m(L,K)$ to denote the distance between linear spaces $L$ and $K$.

\begin{lemma}\label{lem: linear spaces transverse}
    Let $X,Y$ be linear subspaces of $\mathbb{R}^N$ of dimensions $m$ and $n$ respectively.
    Assume that $X$ and $Y$ intersect transversally and let $v \in X \cap Y$.

    Then for any neighborhood $V$ of $v$, there is a 
    neighborhood $W$ of $Y$ so that for any $Y'\in W$
    there exists a $v \in V$ so that $v' \in X\cap Y'$.    
\end{lemma}
\begin{proof}
    Let $x_1,\dots,x_{N-m}$ be a basis of $X^\perp$ and let $y_1,\dots,y_{N-n}$ be a basis of $Y^\perp$.
    Then there exists an open neighborhood $U$ of $Y$ in $\gr(n,\mathbb{R}^N)$ (in the projector metric topology)
    and a continuous function $\phi: U \rightarrow (\mathbb{R}^N)^{N-n}$ such that $\phi(Y) = (y_1,\dots,y_{N-n})$ and
    $\phi(Y')$ is a basis of ${Y'}^\perp$ for each $Y' \in U$.
    If $Y' \in U$, then $X$ and $Y'$ intersect transversely if and only if $\{x_1,\dots,x_{N-m},\phi(Y')_1,\dots,\phi(Y')_{N-n})\}$
    is linearly independent.

    Now let $v \in X \cap Y$.
    Then there exist $p_1,\dots,p_{m+n-N} \in \mathbb{R}^N$ and $b_1,\dots,b_{m+n-N} \in \mathbb{R}$ such that $v$ is the unique solution to the following non-singular system
    \begin{align*}
        \langle x_i, v \rangle = 0 \qquad &{\rm for \ } i = 1,\dots,N-m \\
        \langle y_i, v \rangle = 0 \qquad &{\rm for \ } i = 1,\dots,N-n \\
        \langle p_i, v \rangle = b_i \qquad &{\rm for \ } i = 1,\dots,m+n-N.
    \end{align*}
    Given $Y' \in U$, we may replace the $y_i$'s by $\phi(Y')$ to get another 
    linear system $l(Y')$.
    By continuity of $\phi$, the coefficients in $l(Y')$ will vary continuously with $Y'$.
    So there is a neighborhood   
    $U' \subset U$ of $Y'$ within which
    the linear system will be non-singular
    with solution $v'$.
    Therefore $v'$ will vary continuously with $Y'$. Note that $v' \in X \cap Y'$.
 So we can find a 
    neighborhood $W \subset U'$ that keeps $v'$
    in any arbitrary neighborhood of $v$.
\end{proof}

\begin{lemma}\label{lem: Q bar is continuous}
    The map $\mathcal{S}^n(d) \rightarrow \gr(\binom{n-d+1}{2},V)$ given by $S \mapsto \overline{Q}(S)$ is continuous in the topologies induced by the trace-norm on $\mathcal{S}^n(d)$ and the projector metric on $\gr(\binom{n-d+1}{2},V)$.
\end{lemma}
\begin{proof}
    Let $A \in \mathcal{S}^n$ and consider the map
    \[
        \phi_A:\mathcal{S}^n(d)\rightarrow \mathcal{S}^n(d) \qquad {\rm given \ by} \qquad S \mapsto \pi_{\overline{Q}(S)}A.
    \]
    The entries of $\phi_A(S)$ can be written as rational functions in the entries of $S$, so $\phi_A$ is continuous.
    Fix $\varepsilon > 0$.
    For each $A \in \mathcal{S}^n$, we have that $(\pi_{\overline{Q}(S)}-\pi_{\overline{Q}(T)})(A) = \phi_A(S)-\phi_A(T)$
    so continuity of $\phi_A$ implies that there exists $\delta_A > 0$ be such that if $S,T \in \mathcal{S}^n(d)$
    satisfy $\|S-T\| \le \delta_A$, then $\|(\pi_{\overline{Q}(S)}-\pi_{\overline{Q}(T)})(A)\|  \le \varepsilon$.
    Define $\delta := \inf_{\|A\| = 1} \delta_A$.
    Since each $\delta_A > 0$, compactness of $\{A \in \mathcal{S}^n: \|A \|=1\}$ implies $\delta > 0$.
    Then if $S,T \in \mathcal{S}^n(d)$ satisfy $\|S-T\| \le \delta$ and $A \in \mathcal{S}^n$ satisfies $\|A\| = 1$,
    we have that $\|(\pi_{\overline{Q}(S)} - \pi_{\overline Q(T)})(A)\| \le \varepsilon$.
    By definition of operator norms, this implies $m(\overline{Q}(S),\overline{Q}(T)) = \|\pi_{\overline{Q}(S)} - \pi_{\overline{Q}(T)}\| \le \varepsilon$.
\end{proof}


In the following proposition, transversality allows us,
starting at some pair  $(\Omega_0,S_0)$, 
to perturb $S_0$ in any direction, and to counter this with an appropriate change to 
$\Omega_0$.
On the other hand, once we have reduced the problem to one of transversality, 
it is easy to show (Lemma~\ref{lem:hard2}) that this must happen generically once the dimension of $L$ is 
sufficiently high.

\begin{samepage}
\begin{prop}\label{prop: hard direction}
    Let $m,d,n$ be integers such that $1 \le d \le n$ and $m \ge dn - \binom{d}{2} + 1$
    and let $G \subseteq \gr(m,\mathcal{S}^n)$ consist of all $L$ such that
    \[
        L \cap \mathcal{S}^n_{++} \neq \emptyset.
    \]
    Then there exists a Zariski-closed proper subset $E \subsetneq G$ such that if $L \notin E$ the following holds:
    there exists an open subset $U \subseteq \mathcal{S}^n_{+}(d)$ 
    such that for 
    each matrix $S \in U$, there exists a rank $(n-d)$ PSD matrix $\Omega \in \mathcal{S}^n_{+}\cap L$ such that $\Omega S=0$.
\end{prop}
\end{samepage}
\begin{proof}
Let $E_1, E_2$ denote the Zariski-closed proper subsets of $\gr(m,\mathcal{S}^n)$ indicated in Lemmas~\ref{lem:psd2} and~\ref{lem:hard2}.
Define $E := (E_1 \cup E_2) \cap G$ and assume $L \notin E$.
Lemmas~\ref{lemma:psd1} and~\ref{lem:psd2} give us a matrix $\Omega'_0 \in L$ that is PSD of rank $n-d$.
Lemma~\ref{lem:hard2} then gives us a matrix $\Omega_0 \in L$ that is PSD of rank $n-d$ such that $L$ and $Z(\Omega_0)$ 
intersect transversally.

Now let $M$ be a matrix whose columns are a basis for the kernel of $\Omega_0$ and define 
the PSD matrix $S_0 := MM^T$. By construction, $Q(S_0)=Z(\Omega_0)$.
Lemmas~\ref{lem: linear spaces transverse} and~\ref{lem: Q bar is continuous} then imply that
there is a neighborhood $U$ of $S_0$
so that, if $S\in U$, then $\overline{Q}(S) \cap L$  
will contain a matrix $\Omega$ in any chosen  
neighborhood of $\Omega_0$. 
Due to 
semicontinuity of rank, we have
that $\rank \Omega \ge \rank \Omega_0$.  Since $\Omega\in \overline{Q}(S)$, its rank 
is at most $n - d$, so we have equality of ranks and of signatures.  As 
$S\in U$ was arbitrary, $\Omega\in L\cap \overline{Q}(S)$ is PSD of rank $n - d$, and 
$\Omega S = 0$, completing the proof.
\end{proof}

We now explain how Proposition~\ref{prop: hard direction} gives a lower bound on the maximum likelihood threshold of a linear concentration model $\mathcal{M}_L$.
A generic $L \in \gr(m,\mathcal{S}^n)$ does not lie in the set $E$ in the statement of the proposition.
Thus the proposition implies that there exists a dataset $x_1,\dots,x_d$, with corresponding sample covariance matrix $S$,
and a matrix $\Omega \in L \cap \mathcal{S}^n_+$ such that $\Omega S = 0$,
and that this property is robust with respect to perturbation of $x_1,\dots,x_d$.
Theorem~\ref{thm: existence of optimum criteria} thus implies that the maximum likelihood problem has no optimum
for the dataset $x_1,\dots,x_d$, and that this cannot be fixed by perturbing the dataset.
In other words, there exist generic $d$-element datasets that are not sufficient for solving the maximum likelihood estimation problem.

\begin{ex}
Consider the case where $d=4$, $n=10$ 
and $L$ is the space of $10\times 10$ symmetric matrices whose $ij$ entry is $0$ whenever $i \in \{1,5\}$ and $j \in \{6,10\}$.
This is the linear space corresponding to the Gaussian graphical model on the complete bipartite graph $K_{5,5}$.
In this case, $\dim(L) =35 \ge   nd -\binom{d}{2}+1$. 
Let us consider the set $T$ of 
$S \in \mathcal{S}^n_{+}(d)$ 
    such that there exists a rank $(n-d)$ PSD matrix $\Omega \in \mathcal{S}^n_{+}\cap L$ such that $\Omega S=0$.
    It can be shown using $(d+1)$-connectivity of $K_{5,5}$~\cite{Alfakih-conn}
    and coning~\cite{conic} that
    $T$ is not empty.
    Meanwhile all of the $S\in T$
must satisfy some non-trivial algebraic condition~\cite{Bolker-Roth}. Thus, the conclusion of
Proposition~\ref{prop: hard direction} does
not hold, and this $L$ must be in the exceptional set of that
proposition.
In particular, we see that 
for all PSD $\Omega \in L$ of rank $n-d$,
it must be the case that $Z(\Omega)$ and 
$L$ intersect non-transversally. 
In other words, this $L$ is in the exceptional set described in Lemma~\ref{lem:hard2}.

\end{ex}

\subsection{Upper bound on the MLT}
The main result of this subsection is Proposition~\ref{prop: easy direction}
which tells us that the upper bound from Proposition~\ref{prop: hard direction} on the MLT of $\mathcal{M}_L$ for generic $L$
is also a lower bound. 
This direction is easier since it only requires algebraic considerations, as opposed to semi-algebraic.
Let us define 
\[
    \mathcal{B}(m,d) = 
    \{(L,S) : L\cap \overline{Q}(S) \neq \{0\} \} \subseteq \gr(m,\mathcal{S}^n) \times \mathcal{S}^n(d)
\]
The basic fact underlying this section is that, when $m$ is small relative to $d$, 
$\mathcal{B}(m,d)$ is a proper algebraic subset
of 
$\gr(m,\mathcal{S}^n) \times \mathcal{S}^n(d)$.
\begin{lemma}\label{lem: bundle dimension}
Assume
    \[
        m \le nd - \binom{d}{2}.
    \]
Then $\dim(\mathcal{B}(m,d)) < \dim(\gr(m,\mathcal{S}^n) \times \mathcal{S}^n(d))$.
\end{lemma}
\begin{proof}
$\gr(m,\mathcal{S}^n) \times \mathcal{S}^n(d)$ is an 
irreducible algebraic set.
$\mathcal{B}(m,d)$ is an algebraic subset of 
$\gr(m,\mathcal{S}^n) \times \mathcal{S}^n(d)$,
defined by the vanishing of various determinants.
Since 
$m \le nd - \binom{d}{2}$, for any $S$, almost all of 
$L$ only intersect $\overline{Q}(S)$ at $0$. Thus
$\mathcal{B}(m,d)$ is a proper algebraic subset of $\gr(m,\mathcal{S}^n) \times \mathcal{S}^n(d)$.
A proper algebraic subset of an irreducible variety has
lower dimension.
\end{proof}

\begin{lemma}\label{lem: fibre dimension}
Let $X$ be a complex irreducible 
quasi-projective variety
and $f$ a regular map from $X$ to $\mathbb{C}^N$ for some $N$.
Let $Y$ be the image $f(X)$.
Then there is a Zariski open subset 
$U$ of $Y$ such that  
for $y\in U$, we have 
\[ \dim(X)=\dim(Y)+
\dim(f^{-1}(y)) \]
\end{lemma}
\begin{proof}[Proof sketch]
Let $k$ be the minimal fiber dimension of $f$.  The set of points 
$Z$ in the domain so that $f^{-1}(f(x))$ has dimension larger than 
$k$ is a proper Zariski closed subset \cite[Theorem 11.12]{harris2013algebraic} 
of the irreducible $X$.  
Hence, $f(X\setminus Z)$ is a Zariski dense constructible subset of 
$Y$.  It follows that $f(Z) = Y \setminus f(X\setminus Z)$ is 
contained in a proper algebraic subset $Y'$ of $Y$.  The 
set $Y\setminus Y'$ satisfies the statement of the lemma.
\end{proof}
\begin{remark}
Various reference works make the slightly 
stronger statement that the set of $y\in Y$ so that 
$f^{-1}(y)$ has minimal dimension is open in $Y$.  As 
described in \cite{speyer-example}, this stronger statement is incorrect.
\end{remark}

\begin{lemma}\label{lem: real fibre dimension}
Let $X$ be a proper algebraic subset
of $\RR^M$, 
and $f$ a regular map from $X$ to $\RR^N$ for some $M$ and $N$.
Let $F$ be the (semi-algebraic) subset of $f(X)$ where the fiber dimension equals
$M-N$. Then $\dim(F)<N$.
\end{lemma}
\begin{proof}
Let us first suppose that $X$ is irreducible.  If 
$f(X)$ has dimension less than $N$, we know  
$F \subseteq f(X)$ and the lemma is proved.  Otherwise, 
we complexify $X$ as well as the codomain.
We let $\overline{X}$ be the complex Zariski closure of $X$.
The complex dimenison of $\overline{X}$ is equal to the 
the (real) dimension of $X$.  The complex dimension 
of $Y^* = f(\overline{X})$ is at least the (real) dimension of $f(X)$.
By assumption, $f(X)$ has dimension $N$, and so $f(\overline{X})$ 
must as well.

Define $F^*$ to be the (constructible)
subset of $\mathbb{C}^N$ where the 
fiber dimension of $f$ is $M - N$.  By Lemma 
\ref{lem: fibre dimension}, $F^*$ has (complex) 
dimension less than $N$.  Going back to the real 
$F$, we have  $\overline{F}\subseteq \overline{F^*}$, where 
the closures are in the complex Zariski topology.
Hence, the dimension of $F$ is less than $N$.

If $X$ is reducible, we repeat the argument above for finitely many
irreducible components.  A finite union of set of dimension less than 
$N$ has the same property, so the general case follows.
\end{proof}

\begin{prop}\label{prop: easy direction}
Let $n,d$ be integers such that $1 \le d \le n$ and let $L \subset \mathcal{S}^n$ be a linear subspace.
Let $m$ denote the dimension of $L$ and assume
\[
    m \le nd - \binom{d}{2}.
\]
Then there exists a Zariski-closed proper subset $E \subseteq \gr(m,\mathcal{S}^n)$
such that if $L \notin E$ then there exists a Zariski-closed proper subset $F_L \subset \mathcal{S}^n_0(d)$
such that if $S \in \mathcal{S}^n_0(d) \setminus F_L$ then $\Omega S \neq 0$ for all $\Omega \in L\setminus\{0\}$.
\end{prop}
\begin{proof}
Let $F$ consist of all the $L\in \gr(m,\mathcal{S}^n)$ so that the the fiber of the 
projection from $\mathcal{B}(m,d)$ onto the $L$ coordinate has dimension $\dim \SS^n(d)$.  If the 
projection has image $Y$ with dimension less than $\dim \gr(m,\mathcal{S}^n)$, 
then the proposition follows by taking $E$ to be the Zariski closure of $Y$.

Otherwise, Lemma \ref{lem: bundle dimension} implies that
$\mathcal{B}(m,d)$ is a proper subvariety of ${\gr(m,\mathcal{S}^n)\times\mathcal{S}^n(d)}$.
Lemma \ref{lem: real fibre dimension} then implies $\dim(F)<\dim(\gr(m,\mathcal{S}^n))$.
Hence $\overline{F}$ is of smaller dimension than $\gr(m,\mathcal{S}^n)$
and we may take $E = \overline{F}$ to finish the proof.
\end{proof}




\subsection{Putting it all together}
\begin{thm}\label{thm: main}
    Let $n \in \mathbb{N}$ and let $1 \le m \le \binom{n+1}{2}$.
    Then there exists a Zariski-closed proper subset $E \subset \gr(m,\mathcal{S}^n)$ such that if $L \in \gr(m,\mathcal{S}^n)$ is not in $E$ and
    $L \cap \mathcal{S}^n_{++} \neq \emptyset$,
    then maximum likelihood threshold of $L$ is the minimum $d$ such that
    \[
        m \le nd-\binom{d}{2}.
    \]
\end{thm}
\begin{proof}
    For any $1 \le d \le n$, Theorem~\ref{thm: existence of optimum criteria} implies that the existence of a solution to the optimization problem~\eqref{eq: mle optimization} is equivalent to the non-existence of a nonzero $\Omega \in \mathcal{S}^n_+ \cap L$ such that
    $\Omega S = 0$. 
    First suppose $m \ge nd-\binom{d}{2} + 1$.
    By Proposition~\ref{prop: hard direction}, there exists a Zariski-closed proper subset $E_1 \subset \gr(m,\mathcal{S}^n)$ such that if $L \notin E_1$,
    then there exists an open set $U \subseteq \mathcal{S}^n_+$ such that for all $S \in U$,
    there exists a nonzero $\Omega \in L \cap \mathcal{S}^n_+$ such that $\Omega S = 0$.
    
    Now suppose instead that  $m \le nd-\binom{d}{2}$.
    Proposition~\ref{prop: easy direction} implies existence of a Zariski-closed proper subset $E_2 \subset \gr(m,\mathcal{S}^n)$ such that if $L \notin E_2$,
    then there exists a Zariski-closed proper subset $F_L \subset \mathcal{S}^n_+(d)$
    such that for $S \in \mathcal{S}^n_+(d) \setminus F_L$ 
    (and thus for almost all $S \in \mathcal{S}^n_+(d)$)
    we have $\Omega S \neq 0$ for all nonzero $\Omega \in L$.
    Letting $E$ be the union of the $E_1$
    and $E_2$ over the finite set of applicable
    $d$, we obtain the theorem.
\end{proof}

\begin{remark}\label{remark: main in generic language}
    We can now close the circle of ideas that started with Definition~\ref{defn: genericity}, the definition of genericity.
    The set of linear concentration models, i.e.~the subset of $\gr(d,\mathcal{S}^n)$ consisting of all subspaces that contain a positive definite matrix, is an irreducible semi-algebraic set.
    Thus Theorem~\ref{thm: main} gives the maximum likelihood threshold of a generic linear concentration model.
\end{remark}

We now discuss the ways that some $L \in \gr(m,\mathcal{S}^n)$ containing a positive definite matrix
can have maximum likelihood threshold strictly less than the $d$ given in Theorem~\ref{thm: main}.
If $d$ satisfies $m \ge dn - \binom{d}{2} + 1$, then Proposition~\ref{prop: hard direction}
can fail to apply to $L$ for two reasons:
either $L$ does not contain a PSD matrix of rank $n-d$,
or for every PSD $\Omega \in L$, $Z(\Omega)$ and $L$ intersect non-transversally.
As we saw in Lemmas~\ref{lem:psd2} and~\ref{lem:hard2},
both conditions can be expressed by the vanishing of certain polynomials on the Pl\"ucker coordinates of $L$.
It is interesting to note that Gaussian graphical models that fail to have the expected
maximum likelihood threshold always do so for the second reason.
In particular, if $L \subseteq \mathcal{S}^n$ is defined by setting certain off-diagonal entries to zero, then $L$ always contains a diagonal matrix with $n-d$ nonzero diagonal entries, all of which are positive.

For generic $L$ we see that there is a 
striking gap in behavior at our threshold for $m$.
When 
    \[
        m \le nd-\binom{d}{2},
    \]
then outside of some Zariski-closed
$F_L \subset \mathcal{S}^n_+(d)$,
there is no non-zero 
$\Omega \in L$ \emph{of any signature}
so that $\Omega S=0$.
On the other hand when 
    \[
        m \ge nd-\binom{d}{2}+1
    \]
then there is an open neighborhood 
$U \subset \mathcal{S}^n_+(d)$,
so that for $S \in U$
there is an 
$\Omega \in L$ that is PSD 
\emph{and has rank $n-d$}
and such that $\Omega S=0$. 

The maximal rank $(n-d)$ of $\Omega$ is not needed in order to apply
Theorem~\ref{thm: existence of optimum criteria};
all that is needed is that $\Omega$ is PSD and nonzero.
That said, for generic $L$, once
$m$ is large enough to obtain a non-zero $\Omega$ over such a $U$, we also get an $\Omega$ of rank $n-d$.
For non-generic $L$, such as the ones obtained by setting certain off-diagonal entries to zero, we can get other behavior.
In particular, for some such $L$,
there exist open neighborhoods $U$ of the set of rank-$d$ PSD matrices
where for each $S \in U$, there exist non-zero PSD $\Omega \in L$
with $\Omega S = 0$,
but if $S$ does not lie in a particular Zariski-closed set,
then the PSD $\Omega \in L$ satisfying $\Omega S = 0$ have rank strictly less than $n-d$~\cite[Section 6.1]{bernstein2021maximum}.

\section{Generic completion rank}\label{sec:gcr}
Theorem~\ref{thm: existence of optimum criteria} is a generalization of certain results about \emph{Gaussian graphical models},
which are linear concentration models of a particular type.
We will now define such models and discuss how Theorem~\ref{thm: existence of optimum criteria} fits into what is known about them.
Each graph $G$ on vertex set $\{1,\dots,n\}$ with $e$ edges
defines a linear subspace $L_G$ of $\mathcal{S}^n$ as follows
\[
    L_G := \{A \in \mathcal{S}^n: A_{ij} = 0 \ \text{for all non-edges} \ ij \ {\rm of } \ G\}.
\]
The dimension $m$ of $L_G$ is $e+n$.
The corresponding Gaussian model, denoted $\mathcal{M}_G$, is called the \emph{Gaussian graphical model} associated to $G$.
To our knowledge, such models were first studied by Dempster in~\cite{dempster1972covariance},
where they proved equivalence of the first two conditions of Theorem~\ref{thm: existence of optimum criteria} for Gaussian graphical models.
Uhler found an easily computable upper bound on the maximum likelihood threshold of a given Gaussian graphical model~\cite{uhler2012geometry}.
Blekherman and Sinn called this upper bound the \emph{generic completion rank} in~\cite{blekherman2019maximum} and provided the first known examples where it failed to be sharp.
We now extend the notion of generic completion rank to arbitrary linear covariance models and show that it still 
gives an upper bound on the maximum likelihood threshold.  The proofs of the results here are similar to those of the 
corresponding statements in \cite{uhler2012geometry} and \cite{blekherman2019maximum}.

\begin{defn}[{\cite[Definition~1.1]{blekherman2019maximum}}]\label{defn:generic completion rank}
    The \emph{generic completion rank} of a linear subspace $L \subseteq \mathcal{S}^n$, denoted $\gcr(L)$,
    is the minimum $d$ such that $\pi_L(\mathcal{S}^n_0(d))$ is full-dimensional in $L$.
\end{defn}

When $L = L_G$ for some graph $G$, Definition~\ref{defn:generic completion rank} agrees the definition of generic completion rank given in~\cite{blekherman2019maximum}. Below, we prove Theorem~\ref{thm: gcr is an upper bound for mlt}, which says that generic completion rank upper-bounds the maximum likelihood thresholds, as it does for graphs. Before proving that, we need the following lemma which can be seen as the relaxation of Theorem~\ref{thm: existence of optimum criteria} to allow for non-PSD $\Omega$.
It is applicable to generic completion ranks, as opposed to maximum likelihood thresholds.

The following lemma characterizes the normal space of $\mathcal{S}^n_0(d)$ at $S$ in its embedding in $\mathcal{S}^n$.
This is undoubtedly well-known, but we were unable to find a precise reference using our language so we provide a proof instead.
\begin{lemma}
\label{lem:cnormal}
Let $S \in \mathcal{S}^n_0(d)$ and let $\Omega \in \mathcal{S}^n$.
Then $\Omega S = 0$ if and only if $\tr(\Omega X) = 0$ for all $X \in T_S \mathcal{S}^n_0(d)$.
\end{lemma}
\begin{proof}
Let $W$ denote the smooth submanifold of $\mathbb{R}^{n\times n}$ consisting of \emph{all} $n\times n$ matrices of rank $d$.
Given $M \in W$, the tangent space $T_MW$ consists of all matrices that map the kernel of $M$ into its image~\cite[Example~14.16]{harris2013algebraic}.
Since $S$ is symmetric, its kernel is the orthogonal complement of its span.
Therefore if $A \in T_S W$ and $x,y \in \ker(S)$, the following holds
\begin{equation}\label{equation: trace zero chain of equalities}
    0=y^T A x=tr(y^TAx)=tr(A xy^T).
\end{equation}
The orthogonal complement of a subspace $L$ of $\mathcal{S}^n$ will be denoted by $L^{\perp_s}$,
and that of a subspace $L$ of $\mathbb{R}^n\times \mathbb{R}^n$ will be denoted $L^{\perp_n}$.
We now claim that~\eqref{equation: trace zero chain of equalities} implies the following
\begin{equation}\label{equation: normals spanned by kernel outer products}
    (T_SW)^{\perp_n} = \left\{\sum_{i=1}^{(n-d)^2}\alpha_i x_iy_i^T: x_i,y_i \in \ker(S) \ {\rm for \ all \ } i \ {\rm and \ } \alpha_i \in \mathbb{R}\right\}.
\end{equation}
Indeed, the inclusion $\supseteq$ is clear.
To see the $\subseteq$ inclusion,
first note that $(T_SW)^{\perp_n}$ has dimension $(n-d)^2$.
It now suffices to exhibit $(n-d)^2$ linearly independent matrices of the form $xy^T$ for $x,y \in \ker(S)$.
Let $x_1,\dots,x_{n-d}$ be a basis of $\ker(S)$.
The set of matrices of the form $x_ix_j$ for $1 \le i \le j \le n$ is linearly independent
which can be seen by changing our basis of $\mathbb{R}^n$ so that $x_i$ is the $i^{\rm th}$ standard unit vector.

Viewing $\mathcal{S}^n$ as a subspace of $\mathbb{R}^{n\times n}$, we claim the following:
\begin{equation}\label{equation: tangent space from non-symmetric}
(T_S\mathcal{S}^n_0(d))^{\perp_s} = (T_S W)^{\perp_n}\cap \mathcal{S}^n. 
\end{equation}
The inclusion $(T_S\mathcal{S}^n_0(d))^{\perp_s} \supseteq (T_S W)^{\perp_n} \cap \mathcal{S}^n$ is clear.

To get the reverse inclusion, it now suffices to show that
\[
    \dim(T_SW^{\perp_n} \cap \mathcal{S}^n) \ge \dim((T_S\mathcal{S}^n_0(d))^{\perp_s}) = \binom{n-d+1}{2}.
\]
Let $x_1,\dots,x_{n-d}$ be a basis of $\ker(S)$ and consider the following set $\mathcal{B}$ of $\binom{n-d+1}{2}$ matrices
\[
    \mathcal{B} := \{x_ix_j^T + x_jx_i^T : 1 \le i \le j \le n\}.
\]
By change of basis we can without loss of generality assume that $x_i$ is the $i^{\rm th}$ standard basis vector.
If $i\neq j$ then $x_ix_j^T + x_jx_i^T$ is the matrix where every entry is zero aside from $(i,j)$ and $(j,i)$ which are both $1$.
If $i = j$, then $x_ix_j^T + x_jx_i^T$ is the matrix that is all zeros aside from a $2$ at the $(i,i)$ entry.
Thus $\mathcal{B}$ is linearly independent and~\eqref{equation: normals spanned by kernel outer products} implies
that $\mathcal{B} \subset (T_SW)^{\perp_n} \cap \mathcal{S}^n$.

In light of~\eqref{equation: tangent space from non-symmetric}, the lemma follows if we prove the 
the following:
\[
    \{\Omega \in \mathcal{S}^n : \Omega S = 0\} = (T_S W)^{\perp_n} \cap
    \mathcal{S}^n 
 \]
Suppose 
$\Omega \in (T_S W)^{\perp_n} \cap \mathcal{S}^n$ and define $k:=\rank(\Omega)$.
Then $\Omega$ has a spectral decomposition $\Omega = U\Lambda U^T$ where $U$ is an $n\times k$ matrix with orthonormal columns
and $\Lambda$ is a nonsingular diagonal $k\times k$ matrix.
From~\eqref{equation: normals spanned by kernel outer products} it follows that $\Omega S= 0$.

Now suppose $\Omega \in \mathcal{S}^n$ and $\Omega S =0$.
As before, we have a spectral decomposition $\Omega = U\Lambda U^T$ where $U$ has linearly independent columns so $U^T S = 0$.
Thus the columns of $U$ are in the kernel of $S$.
From its spectral decomposition, we see that 
$\Omega$ can be written as a sum of outer products
of vectors in the kernel of $S$ as in~\eqref{equation: normals spanned by kernel outer products}.
This implies that $\Omega \in (T_S W)^\perp$.
\end{proof}

\begin{lemma}\label{lem:sub}
    Let $S\in \mathcal{S}^n_0(d)$ and 
    let $L \subseteq \mathcal{S}^n$ be a linear space.
    Then there exists a nonzero $\Omega \in L$ such that $\Omega S = 0$
    if and only if $\pi_L: \mathcal{S}^n_0(d) \rightarrow L$ is not a submersion at $S$.
\end{lemma}
\begin{proof}
    In light of Lemma~\ref{lem:cnormal},
    the existence of some $\Omega \in L\setminus\{0\}$ such that $\Omega S = 0$
    is equivalent to the following
    \[
        \dim((L) \cap (T_S \mathcal{S}^n_0(d))^{\perp}) \ge 1
    \]
    which is equivalent to 
    $\dim(\pi_L(T_S\mathcal{S}^n_0(d))) < \dim L$ which is what it means for $\pi_L: \mathcal{S}^n_0(d) \rightarrow L$
    to not be a submersion at $S$.
\end{proof}

\begin{lemma}\label{lemma: full dim proj iff no stress}
    Let $L \subseteq \mathcal{S}^n$ be a linear subspace.
    Then $\pi_L(\mathcal{S}^n_0(d))$ is full-dimensional in $L$ if and only if for every regular point $S \in \mathcal{S}^n_0(d)$
    of the restriction of $\pi_L$ to $\mathcal{S}^n_0(d)$, if $\Omega \in L$ and $\Omega S = 0$ 
    then $\Omega = 0$.
\end{lemma}
\begin{proof}
    First assume that $\pi_L(\mathcal{S}^n_0(d))$ is full dimensional in $L$ and that $S$ is a regular point of the restriction of $\pi_L$ to $\mathcal{S}^n_0(d)$.
    Then Sard's Theorem~\cite{guillemin2010differential} implies that $\pi_L(T_S \mathcal{S}^n_0(d)) = L$.     
    Let $\Omega \in L$ satisfy $\Omega S = 0$,
    let $Y \in L$ and let $X \in T_S \mathcal{S}^n_0(d)$ satisfy $\pi_L(X) = Y$.
    Lemma~\ref{lem:cnormal} implies that $\tr(\Omega X) = 0$.
    Since orthogonal projection in an inner product space is equivariant with respect to the inner product,
    this implies that $\tr(\pi_L(\Omega)\pi_L(X)) = 0$.
    Since $\Omega \in L$ and $\pi_L(X) = Y$, this gives $\tr(\Omega Y) = 0$.
    We have thus proven that $\Omega \in L^\perp$.
    Since $\Omega \in L$ as well, $\Omega = 0$.

    Now assume that for every regular point $S$ of the restriction of $\pi_L$ to $\mathcal{S}^n_0(d)$,
    if $\Omega \in L$ and $\Omega S = 0$ then $\Omega = 0$.
    Let $S$ be a regular point of this restriction.
    Lemma~\ref{lem:cnormal} implies that $L \cap (T_S \mathcal{S}^n_0(d))^\perp = 0$
    and thus $L^\perp + T_S\mathcal{S}^n_0(d) = \mathcal{S}^n$.
    Therefore $\pi_L(T_S\mathcal{S}^n_0(d)) = L$.
    Then the submersion theorem~\cite{guillemin2010differential} implies that
    $\pi_L(\mathcal{S}^n_0(d))$ is full-dimensional in $L$.
\end{proof}

\begin{thm}\label{thm: gcr is an upper bound for mlt}
    Let $L \subseteq \mathcal{S}^n$ be a linear subspace.
    Then $\gcr(L)$ is the minimum $d$ such that for all generic PSD $S$ of rank $d$,
    there does not exist an $\Omega \in L$ such that $\Omega S = 0$.
    Moreover, $\gcr(L) \ge \mlt(L)$.
\end{thm}
\begin{proof}
    The first statement follows from Lemma~\ref{lemma: full dim proj iff no stress},
    taking the non-generic locus to be the singular points of $\pi_L$.
    The second statement now follows from part 3 of Theorem~\ref{thm: existence of optimum criteria}.
\end{proof}

\begin{cor}\label{cor: mlt is gcr for generic}
    The maximum likelihood threshold of a generic linear subspace $L \subseteq \mathcal{S}^n$ is equal to its generic completion rank.
\end{cor}
\begin{proof}
    Given Theorem~\ref{thm: gcr is an upper bound for mlt},
    Theorem~\ref{thm: main} implies that the generic completion rank of a generic $L \in \gr(m,\mathcal{S}^n)$ is at least the minimum $d$ such that $m \ge dn-\binom{d}{2}$ and
    Proposition~\ref{prop: hard direction} implies that the generic completion rank is at most this $d$.
\end{proof}

\section*{Acknowledgments}
LST  was partially supported by UK Research and Innovation 
(grant number UKRI1112), under the EPSRC Mathematical Sciences 
Small Grant scheme.  This work was initiated during a visit by DIB 
to St.~Andrews supported by the Heilbronn Institute for Matheamtical 
Research (HIMR).

\bibliographystyle{abbrvnat}
\bibliography{linear}
\end{document}